\documentclass[sn-mathphys-num,12pt,3p]{sn-jnl}

\usepackage{graphicx}%
\usepackage{multirow}%
\usepackage{amsmath,amssymb,amsfonts}%
\usepackage{amsthm}%
\usepackage{mathrsfs}%
\usepackage[title]{appendix}%
\usepackage{xcolor}%
\usepackage{textcomp}%
\usepackage{manyfoot}%
\usepackage{booktabs}%
\usepackage{algorithm}%
\usepackage{algorithmicx}%
\usepackage{algpseudocode}%
\usepackage{listings}%

\theoremstyle{thmstyleone}%
\newtheorem{theorem}{Theorem}
\newtheorem{proposition}[theorem]{Proposition}%

\theoremstyle{thmstyletwo}%
\newtheorem{example}{Example}%
\newtheorem{remark}{Remark}%
\theoremstyle{thmstylethree}%
\newtheorem{definition}{Definition}%

\begin{document}
	
	\title[Some Sharp bounds for the generalized Davis--Wielandt radius]{Some Sharp bounds for the generalized Davis--Wielandt radius}
	
	
	\author*[1]{\fnm{Mehdi} \sur{Naimi}}\email{mehdi.naimi@univ-mosta.dz}
	
	\author[2]{\fnm{Mohammed} \sur{Benharrat}}\email{mohammed.benharrat@enp-oran.dz}

	\author[1]{\fnm{Faouzi} \sur{Hireche}}\email{faouzi.hireche@univ-mosta.dz}

	\affil[1]{ \orgname{Universit\'e  Abdelhamid Ibn Badis de Mostaganem}, \orgaddress{\postcode{27000}, \state{Mostaganem}, \country{Alg\'erie}}}
	
	\affil[2]{ \orgname{Ecole Nationale Polytechnique d'Oran-Maurice Audin (Ex. ENSET d'Oran)}, \orgaddress{\street{BP 1523 Oran-El M'naouar},  \postcode{ 31000}, \state{Oran}, \country{Alg\'{e}rie}}}

	
	\abstract{This paper presents a study of the generalized Davis--Wielandt radius of Hilbert space operators. New sharp lower bounds and improvements for both the generalized Davis--Wielandt radius and the numerical radius are established, supported by illustrative examples. In addition, an alternative form of the triangle inequality for operators is derived.}

	
	
	\maketitle
	\section{Introduction}

	Let $\mathcal{H}$ be a complex Hilbert  space with inner product $\langle \cdot,\cdot \rangle$ and the corresponding norm $ \Vert \cdot \Vert$,  and let $B(\mathcal{H})$ denote the $C^*$-algebra of bounded linear operators on $\mathcal{H}$. For $T \in B(\mathcal{H})$, let $T^*$ be its adjoint and set $|T| = (T^*T)^{1/2}$, $|T^*| = (TT^*)^{1/2}$.
	The numerical range and the numerical radius  of $T$ are  defined by
	$$W(T) :=\lbrace \langle Tx,x\rangle : x \in \mathcal{H}, \Vert x \Vert  = 1\rbrace $$
	and
	$$ w(T):=\sup\lbrace \vert z\vert : z \in W(T)\rbrace   $$
	respectively. It is well known that $w(\cdot)$ defines a norm on $B(\mathcal{H} )$, which is equivalent to the usual  operator norm, more precisely, we have  
	\begin{equation}\label{norm.equiv}
		\frac{1}{2}\Vert T\Vert \leq w(T) \leq \Vert T\Vert
	\end{equation}
	for any $T\in   B(\mathcal{H} )$. We refer the reader to \cite{Abu,Bak,Kar,Kit1,Yam,BhuBook} for further properties and
	bounds of the numerical radius. A nice definition of  $w(\cdot)$, is given by Yamazaki in \cite{Yam} as follows
	\begin{equation}\label{eq1Yam}
		w(T)= \sup_{\theta \in \mathbb{R}}\left\| \Re (e^{i\theta}T) \right\|,
	\end{equation}
	for any $T\in B(\mathcal{H})$, where $\Re(T)= \frac{1}{2}(T +T^*)$ and $\Im(T)=\frac{1}{2i}(T -T^*)$ are the real  and the imaginary parts    of $T$, respectively.
	
	Due to the relevance of the numerical radius, numerous generalizations have been  investigated  in the literature, see \cite{Abu,Alo1,Haj,Kit,Sad}. Once remarkable extension is the Davis--Wielandt radius \cite{Dav,Wie}, which is defined by 
	\begin{equation}
		d{w}(T) =  \sup_{x \in \mathcal{H},\|x||=1}  {\sqrt{\left| \langle Tx,x\rangle \right|^2  +\left\| Tx\right\|^4}}.
	\end{equation}
	Obviously, $d{w}(T) \geq 0$ and $d{w}(T)=0$ if and only if $T=0$. However  $d{w}(\cdot)$ fails to satisfy the triangular inequality. However, Bhunia et al.  have  given in \cite{Bhu} the following inequalities
	\begin{align}\label{trian}
		d{w}(T +S ) & \leq  \sqrt{2(d{w}^2(T ) + d{w}^2(S )) + 6\left\| |T|^4 + |S|^4 \right\| } \nonumber\\
		&  \leq 2\sqrt{2}(dw(T ) + dw(S ))
	\end{align} 
	for $T$ and $S$ in $B(\mathcal{H})$. Clearly, $d{w}(T)$ is not a norm on $B (\mathcal{H})$. In spite of this, it has several important properties, see \cite{Zam2,Bhu---r,Bha1}. The next inequality follows directly from the
	definition of $dw(T)$:
	\begin{equation}\label{dw1}
		\max \{ w(T), \left\| |T| \right\|^2  \}\leq d{w}(T) \leq   \sqrt{{w^2(T) + \left\| |T| \right\|^4 }}. 
	\end{equation}
	Recently, based on \eqref{eq1Yam}, Abu-Omar et al. in \cite{Abu}  have presented a generalized definition  of the numerical radius by
	\begin{equation}
		w_N(T)= \sup_{\theta \in \mathbb{R}}N \left( \Re (e^{i\theta}T) \right) ,
	\end{equation}
	with $N(\cdot)$ is an arbitrary norm on $B(\mathcal{H})$. Recall that a norm  $N(\cdot)$ is said to be self--adjoint if $N(T)=N(T^*)$ for every $T\in B(\mathcal{H})$ and algebra norm if $N(TS) \leq N(T)N(S)$ for $T,S \in B(\mathcal{H})$. In \cite{Abu}, it was shown that $w_N(\cdot)$ is a self--adjoint norm and 
	\begin{equation}
		\dfrac{1}{2}\max\{ N(T),N(T^*) \} \leq w_N(T) \leq  \dfrac{1}{2}\left( N(T)+N(T^*)\right).
	\end{equation}
	In particular, if $N (\cdot)$ is self--adjoint, then 
	\begin{equation}
		\dfrac{1}{2} N(T) \leq w_N(T) \leq  N(T).
	\end{equation} 
	Further,  In \cite{Bak} it is shown that
	\begin{equation}
		w_N(T) \geq  \dfrac{1}{2} \sqrt{|N(|T|^2  + |T^*|^2)  - 2w_N(T^2)|}\, ,
	\end{equation}
	when $N(\cdot)$ is a self--adjoint algebra norm.
	
	The reader is referred to \cite{Abu,Bak,Bot} for further details about $w_N(\cdot)$. 
	
Following the idea of the generalized numerical radius introduced in~\cite{Abu}, a similar extension was proposed in~\cite{Alo} for the Davis--Wielandt radius. For clarity, we recall the definition from~\cite{Alo} below:
	\begin{definition}
		Let $T \in B(\mathcal{H})$ and $N(\cdot)$ be a norm, the generalized Davis--Wielandt radius $dw_N(T)$ is given by
		\begin{equation}\label{def*}
			dw_{N}(T) = \sup_{\theta \in \mathbb{R}}  \sqrt{N^2(\Re (e^{i\theta}T))+N^4(\Re(e^{i\theta}|T|))}.
		\end{equation}	
	\end{definition}
	Unlike the generalized numerical radius $ w_N(\cdot) $, which reduces to the classical numerical radius when $ N(\cdot) $ is the usual operator norm, the generalized Davis--Wielandt radius $ dw_N(\cdot) $ does {not} coincide with the classical Davis--Wielandt radius in that case. Still, this definition is useful because it shares many interesting properties with the classical case, as shown in \cite{Alo} and in the present work.
	
	It can be observed that $dw_{N}(T)$ is not a norm and does not satisfy the triangular inequality. Some interesting estimation of $dw_{N}(T)$ are, however, given in \cite{Alo}, in particular, we have
	\begin{equation}\label{equiv}
		\max \{ w_N(T),  w^2_N(|T| ) \}\leq d{w}_N(T) \leq   \sqrt{{w^2_N(T) +  w^4_N(|T| ) }}.
	\end{equation}
	Furthermore, in \cite[Theorem 3]{Alo} the following estimation is provided
	\begin{equation}\label{low}
		dw_{N}(T) \geq    \sqrt{\dfrac{ 1}{4}N^2(T)+\dfrac{ 1}{8}N^4(|T|) }.
	\end{equation} 
	Moreover, the next upper bound is given
	\begin{equation}\label{up}
		dw_{N}(T) \leq   \inf_{\theta \in \mathbb{R}} \sqrt{N^2(\Re(e^{i\theta}T))+N^2(\Im(e^{i\theta}T)) + N^4(\Re(e^{i\theta}|T|))}.
	\end{equation} 
	However, equality \eqref{up} does not hold in general. Indeed,
	a simple counterexample is obtained by choosing the identity operator $T=I$ and the usual norm $N(\cdot) = \| \cdot\|$. Through a straightforward calculation using \eqref{def*}, it can be shown that $dw_{\| \cdot\|}(I)= \sqrt{2}$, and by choosing ${\theta}=\frac{\pi}{2}$ in \eqref{up}, we have 
	\begin{align*}
		dw_{\| \cdot\|}(I) &\leq   \inf_{\theta \in \mathbb{R}} \sqrt{\|\Re(e^{i\theta}I)\|^2 +\|\Im(e^{i\theta}I)\|^2 + \| \Re(e^{i\theta}|I|)\|^4}\\
		& \leq \sqrt{\|\Re(iI)\|^2 +\|\Im(iI)\|^2 + \| \Re(i|I|)\|^4}\\
		& \leq  \sqrt{\|I\|^2 }=1.
	\end{align*}

	This paper sets out to prove sharp bounds for the generalized Davis-Wieland inequality and  illustrate it by  examples. In addition,  we prove an alternative version of the triangular inequality.    
	
	\section{Main Results}    
	We begin this section with a result concerning the equality of the generalized Davis--Wielandt radius of bounded linear operators.
	\begin{proposition}\label{proposit-1}
		Let  $T \in B(\mathcal{H} ) $. We have
		\begin{enumerate}
			\item[(i)] If $dw_{N}(T) = w_N(T)$. Then $T=0$ or $w_N(T)= \max\{N(\Re(iT)),N(\Re(-iT))\}$;
			\item[(ii)] If $dw_{N}(T) = N^2(|T|)$. Then $T$ is anti--Hermitian ($T=-T^*$).
		\end{enumerate}
	\end{proposition}
	\begin{proof}
		(i) Take $T \in B(\mathcal{H} )$, and $\theta_0$ such that $w_N(T) =N(\Re ({e^{i\theta_0}T})) $, we have
		\begin{align*}
			dw_{N}^2(T) &= \sup_{\theta \in \mathbb{R}}  \left( {N^2(\Re (e^{i\theta}T))+N^4(\Re(e^{i\theta}|T|))}\right)\\
			& \geq N^2(\Re (e^{i\theta_0}T))+N^4 \left( \dfrac{e^{i\theta_0}|T| +e^{-i\theta_0}|T|}{2} \right)\\
			& = w_N^2(T) + \cos^4(\theta_0) N^4 \left( |T| \right),
		\end{align*}	
		which implies that $\cos^4(\theta_0)N^4 \left( |T| \right) \leq 	dw_{N}^2(T) - w_N^2(T)$.\\
		By assumption, we obtain $\cos^4(\theta_0) N^4 \left( |T| \right) =0$, then either $N^4 \left( |T| \right)=0$ which gives $T=0$ or  $\cos^4(\theta_0)=0$ which implies that $e^{i\theta_0}=\pm i$ and so  $w_N(T) =N(\Re ({iT}))$ or $w_N(T) =N(\Re ({-iT}))$.\\
		(ii) It is easy to see that  
		\begin{align*}
			dw_{N}^2(T) &= \sup_{\theta \in \mathbb{R}}  \left( {N^2(\Re (e^{i\theta}T))+N^4(\Re(e^{i\theta}|T|))}\right)\\
			& \geq N^2(\Re (T))+ N^4 \left( |T| \right) \text{ \quad ( by taking $\theta=0$)},
		\end{align*}	
		since $dw_{N}(T) = N^2(|T|)$, then  $N(\Re (T))=0 $. So we conclude that $T$ is anti--Hermitian.
	\end{proof}
\begin{remark}
The converse of Proposition~\ref{proposit-1} is not true. We illustrate this fact by providing two counterexamples, considering the usual operator norm
\( N(\cdot) = \|\cdot\| \).
	
	\begin{itemize}
		\item[(i)] Let $T=2iI$. We know that  $w_{\Vert \cdot \Vert }(T)$ coincides with the classical numerical radius
		$w(T)$. In this case,
		\[
		w(T)=\max\{\|\Re(iT)\|,\|\Re(-iT)\|\}=2.
		\]
		However,
		\begin{align*}
			dw_{\|\cdot\|}(T)
			&= \sup_{\theta\in\mathbb{R}}
			\sqrt{\|\Re(e^{i\theta}T)\|^2 + \|\Re(e^{i\theta}|T|)\|^4 }  \\
			&\ge \sqrt{\|\Re(T)\|^2 + \|\Re(|T|)\|^4} \quad \text{(by taking $\theta = 0$)}  \\
			&= \|T\|^2 = 4,
		\end{align*}
		and therefore
		\[
		dw_{\|\cdot\|}(T) > w(T).
		\]
		
		\item[(ii)] Let $T=\frac{i}{2}I$. Then $T$ is anti--Hermitian and
		\[
		N^2(|T|)=\|T\|^2=\frac14.
		\]
		Nevertheless,
		\begin{align*}
			dw_{\|\cdot\|}(T)
			&= \sup_{\theta\in\mathbb{R}}
			\sqrt{\|\Re(e^{i\theta}T)\|^2 + \|\Re(e^{i\theta}|T|)\|^4 }\\
			&\ge w(T)
			= \frac12,
		\end{align*}
		which implies
		\[
		dw_{\|\cdot\|}(T) > \|T\|^2.
		\]
	\end{itemize}
\end{remark}

	In the upcoming theorem, we establish a stronger inequality than the one presented in \cite[Theorem 3]{Alo}. We need to observe that, for any two real numbers $a$ and $b$, we can write
	$$\max\{ a, b\}= \dfrac{a+b + |a-b|}{2}.  $$

	\begin{theorem}\label{The2}
		Let $T \in B(\mathcal{H} ) $ and $N(\cdot)$ be a norm. Then  
		\begin{equation}\label{equa1} 
			dw_{N}(T) \geq   \dfrac{ 1}{2} \sqrt{N^2(T)+2N^4(|T|) + 2\left|N^2(\Re (T))+N^4(|T|) -  N^2(\Im (T)) \right|}.
		\end{equation} 
	\end{theorem}
	\begin{proof} 
		We know that 
		\begin{align*}
			dw_{N}(T) =& \sup_{\theta \in \mathbb{R}}  \sqrt{N^2(\Re (e^{i\theta}T))+N^4(\Re(e^{i\theta}|T|))}\\
			\geq & \sqrt{N^2(\Re (e^{i\theta}T))+N^4(\Re(e^{i\theta}|T|))} \text{, for $\theta \in \mathbb{R}$. }
		\end{align*} 
		So, by taking respectively $\theta =0 $ and $\theta =-\frac{\pi}{2} $, we get:
		\begin{equation*} 
			dw^2_{N}(T) \geq   {N^2(\Re (T))+N^4(|T|) }
		\end{equation*} 
		and 
		\begin{equation*} 
			dw^2_{N}(T) \geq   {N^2(\Im (T)) }. 
		\end{equation*}
		Using, the fact that 
		\begin{align}\label{eqt1} 
			N^2(\Re (T)) + N^2(\Im (T)) &  \geq \dfrac{1}{2} \left(  N(\Re (T)) + N(\Im (T))\right)^2 \nonumber \\
			& \geq \dfrac{1}{2} N^2(\Re (T) + i \Im (T) ) \nonumber\\
			& \geq \dfrac{1}{2} N^2(\Re (T) + i \Im (T) ) \nonumber \\
			& \geq \dfrac{1}{2} N^2(T),
		\end{align}	
		we have
		\begin{align*}
			dw^2_{N}(T) \geq  & \max \{  N^2(\Re (T))+N^4(|T|) , N^2(\Im (T)) \}  \\
			= &  \dfrac{1}{2}\bigg(  N^2(\Re (T))+ N^2(\Im (T)) + N^4(|T|)  \\
			+ &  \left|N^2(\Re (T))+N^4(|T|) -  N^2(\Im (T)) \right|  \bigg) \\
			\geq  & \dfrac{1}{4} N^2(T) +\dfrac{1}{2} N^4(|T|)+ \dfrac{1}{2} \left|N^2(\Re (T))+N^4(|T|) -  N^2(\Im (T)) \right|,
		\end{align*} 
		from which the result follows.
	\end{proof}
\begin{remark}
	The lower bound in Theorem~\ref{The2}  improves inequality~\eqref{low} stated in \cite[Theorem 3]{Alo}. Moreover, as immediate consequences of Theorem~\ref{The2}, we obtain the following simplified bounds:
	\begin{itemize}
		\item For any norm \( N(\cdot) \),
		\[
		dw_N(T) \geq \frac{1}{2}\sqrt{N^2(T)+2N^4(|T|)}.
		\]
		\item If \( N(\cdot) \) is an algebra norm, then using submultiplicativity we further get
		\[
		dw_N(T) \geq \frac{1}{2}\sqrt{N(T^2 + 2|T|^4)}.
		\]
	\end{itemize}
\end{remark}
	\begin{remark}
	Let $T \in B(\mathcal H)$ and let $N(\cdot)=\|\cdot\|$ be the usual operator norm.
	Assume that the numerical range $W(T)$ is a centered disc. Then
	\[
	\left\|\Re\!\left(e^{i\theta}T\right)\right\| = w(T),
	\qquad \forall\,\theta\in\mathbb R.
	\]
	Therefore
	\begin{align}\label{centred disc numerical range}
		dw_{ \Vert \cdot \Vert}(T) = & \sup_{\theta \in \mathbb{R}}   \sqrt{{\left\| \Re (e^{i\theta}T) \right\|^2  +\left\| \Re(e^{i\theta}|T| ) \right\|^4} }\nonumber \\				
	=	&  \sup_{\theta \in \mathbb{R}}   \sqrt{{w^2(T)   +\left\| \Re(e^{i\theta}|T| ) \right\|^4} }  \nonumber\\
	=	&  \sup_{\theta \in \mathbb{R}}   \sqrt{{w^2(T)   + \cos^4(\theta)\left\| |T|  \right\|^4} }  \nonumber\\
	=&  {\sqrt{w^2(T)  +\left\| T \right\|^4}}.
	\end{align}		
	\end{remark}
The identity \eqref{centred disc numerical range} does not hold, in general, for the classical Davis--Wielandt radius $dw(\cdot)$. 
To illustrate this fact, we propose the following example.
		\begin{example}\label{nilp}
		Let $\| \cdot\|$ be  the usual norm and $T=\left( \begin{matrix}
			0 & 2\\ 
			0&0\\
		\end{matrix}\right) $, then 
		\begin{align*}
			dw^2_{ \Vert \cdot \Vert}(T) = & \sup_{\theta \in \mathbb{R}}  \left( {\left\| \Re (e^{i\theta}T) \right\|^2  +\left\| \Re(e^{i\theta}|T| ) \right\|^4} \right) \\				
			=	&  \left\| \Re (T) \right\|^2  +\left\| |T| \right\|^4  \\
			= &  17.
		\end{align*}
		Also, by \eqref{centred disc numerical range}, we have \begin{align*}
			{{w^2(T)  +\left\| T \right\|^4}}= 1^2+2^4=17.
		\end{align*}
		
		However, the generalized  Davis--Wielandt radius $dw_{ \Vert \cdot \Vert}(\cdot)$ associated with the usual norm doesn't align perfectly with Davis--Wielandt radius $dw(\cdot)$. Indeed,	let  $z=\left( \begin{matrix}
			x\\ 
			y\\
		\end{matrix}\right) $, then we get
		\begin{align*}
			d{w}^2(T) = & \sup_{|z|=1}  \left( {\left\| \langle Tz,z\rangle \right\|^2  +\left\| Tz \right\|^4} \right) \\				
			=	& \sup_{|x|^2 + |y|^2 =1} (4|x|^2|y|^2 + 16|y|^4) \\
			=  & 16,
		\end{align*}
			and therefore 
			$$d{w}(T)  \ne {\sqrt{w^2(T)  +\left\| T \right\|^4}}. $$
	\end{example}
	\begin{remark}
		Let $T \in B(\mathcal{H} ) $ and $N(\cdot)$ a norm, it is easy to see that 
		
		\begin{equation*}
			dw_{N}(T) = \sup_{\theta \in \mathbb{R}}\sqrt{N^2(\Im (e^{i\theta}T))+N^4(\Im(e^{i\theta}|T|))}.
		\end{equation*}
		
	\end{remark}
	
	Now, we establish our second theorem, which offers a lower bound for the generalized Davis--Wielandt radius of bounded operators.
	\begin{theorem}\label{alg-nor-res}
		Let $T \in B(\mathcal{H} ) $ and $N(\cdot)$ be an algebra norm. Then

		\begin{equation}\label{eqb1}
			dw_{N}(T) \geq   \dfrac{ 1}{2} \sqrt{N(|T|^2  + |T^*|^2  ) + 2N^4(|T|)+ 2\left|N^2(\Re (T))+N^4(|T|) -  N^2(\Im (T)) \right|},
		\end{equation} 
		and 
		\begin{equation}\label{eqb2}
			dw_{N}(T) \geq   \dfrac{ 1}{2} \sqrt{N(T^2  + T^{*2}  ) + 2N^4(|T|)+ 2\left|N^2(\Re (T))+N^4(|T|) -  N^2(\Im (T)) \right|}.
		\end{equation}
	\end{theorem}
	\begin{proof} 
		We know that 
		\begin{align*}
			dw_{N}(T) =& \sup_{\theta \in \mathbb{R}}  \sqrt{N^2(\Re (e^{i\theta}T))+N^4(\Re(e^{i\theta}|T|))}\\
			\geq & \sqrt{N^2(\Re (e^{i\theta}T))+N^4(\Re(e^{i\theta}|T|))} \text{, for $\theta \in \mathbb{R}$. }
		\end{align*} 
		So, by taking respectively $\theta =0 $ and $\theta =-\frac{\pi}{2} $, we get:
		
		\begin{equation*} 
			dw^2_{N}(T) \geq   {N^2(\Re (T))+N^4(|T|) }
		\end{equation*} 
		and 
		\begin{equation*} 
			dw^2_{N}(T) \geq   {N^2(\Im (T)) }. 
		\end{equation*}
		Now, if   $N(\cdot)$ is an algebra norm, then  we have 
		\begin{align}\label{eqt2} 
			N^2(\Re (T)) + N^2(\Im (T)) &  \geq N(\Re (T)^2) + N(\Im (T)^2) \nonumber \\
			& \geq  N(\Re (T)^2 + \Im (T)^2) \nonumber\\
			& = \dfrac{1}{2} N(|T|^2  + |T^*|^2  ).
		\end{align}	
		Thus 
		\begin{align*}
			dw^2_{N}(T) \geq  & \max \{  N^2(\Re (T))+N^4(|T|) , N^2(\Im (T)) \}  \\
			= &  \dfrac{1}{2}\left(  N^2(\Re (T))+ N^2(\Im (T)) + N^4(|T|)  \right. \\
			+ & \left. \left|N^2(\Re (T))+N^4(|T|) -  N^2(\Im (T)) \right|  \right) \\
			\geq  & \dfrac{1}{4} N(|T|^2  + |T^*|^2  ) +\dfrac{1}{2} N^4(|T|)+ \dfrac{1}{2} \left|N^2(\Re (T))+N^4(|T|) -  N^2(\Im (T)) \right|,
		\end{align*} 
		as required. \\
		Also, we have 
		\begin{align}\label{eqt*2} 
			N^2(\Re (T)) + N^2(\Im (T)) &  \geq N(\Re (T)^2) + N(\Im (T)^2) \nonumber\\
			& \geq  N(\Re (T)^2 - \Im (T)^2)\nonumber \\
			& = \dfrac{1}{2} N(T^2  + T^{2*}  ).
		\end{align}	
		Thus
		\begin{align*}
			dw^2_{N}(T) \geq \dfrac{1}{4} N(T^2  + T^{*2}  ) +\dfrac{1}{2} N^4(|T|)+ \dfrac{1}{2} \left|N^2(\Re (T))+N^4(|T|) -  N^2(\Im (T)) \right|,
		\end{align*} 
		this completes the proof.
	\end{proof} 
	The following example shows that the lowers bounds mentioned above (\eqref{eqb1}, \eqref{eqb2}) for $dw_{N}(\cdot)$
	are sharps.
	\begin{example}
		If $P$ is a non-null orthogonal projection and $\| \cdot\|$ is the usual norm, then 
		\begin{align*}
			dw^2_{ \Vert \cdot \Vert}(P) = & \sup_{\theta \in \mathbb{R}}  \left( {\left\| \Re (e^{i\theta}P) \right\|^2  +\left\| \Re(e^{i\theta}|P| ) \right\|^4}\right)  \\				
			=	& 2 \\
			= &\dfrac{1}{4} \left(  \left\|   |P|^2  + |P^*|^2 \right\| +  2\left\|  |P|\right\| ^4\right.   \\
			+& \left.  2\left|\left\|  \Re (P)\right\| ^2  +\left\| |P|\right\| ^4  -  \left\|  \Im (P)\right\|^2  \right| \right) \text{ \quad by \eqref{eqb1}} \\
			= & \dfrac{1}{4} \left(   \left\|   P^2  + P^{*2} \right\|   + 2\left\|  |P|\right\| ^4  \right. \\ 
			+& \left. 2\left|\left\|  \Re (P)\right\| ^2  +\left\| |P|\right\| ^4  -  \left\|  \Im (P)\right\|^2  \right|\right)  \text{ \quad by \eqref{eqb2}}. 
		\end{align*} 
	\end{example} 
	

	Since $|T|$ is self--adjoint operator, we obtain the following interesting inequalities.
	\begin{remark}
		Let $T \in B(\mathcal{H} ) $ and $N(\cdot)$ be an algebra norm. it is easy to see from Theorem \ref{alg-nor-res} that

		\begin{equation}
			dw_{N}(T) \geq   \dfrac{ 1}{2} \sqrt{w_N(|T|^2  + |T^*|^2  ) + 2 w_N^4(|T|)}
		\end{equation} 
		and	
		\begin{equation}
			dw_{N}(T) \geq   \dfrac{ 1}{2} \sqrt{w_N(|T|^2  + |T^*|^2   +2|T|^4)}.
		\end{equation} 
		
	\end{remark}
	
	The following estimation for the generalized numerical radius  is an  improvement of that given in  \cite[Theorem 5]{Bak}.
	\begin{theorem}
		Let $T \in B(\mathcal{H} ) $ and $N(\cdot)$ an algebra norm. Then 
		\begin{align*}
			w_N(T) \geq  \dfrac{1}{2} \max \left\lbrace \sqrt{N(|T|^2  + |T^*|^2  )}, \sqrt{N( T^2  + T^{*2}  )}  \right\rbrace. \
		\end{align*} 
	\end{theorem}
	\begin{proof} 
		For any $\theta \in \mathbb{R}$, we have 
		\begin{align*}
			2w^2_N(T) & \geq 
			N^2(\Re (e^{i\theta}T)) + N^2(\Im (e^{i\theta}T))\\
			& \geq  \dfrac{1}{2} \max \left\lbrace N(|T|^2  + |T^*|^2  ), N( T^2  + T^{*2}  )  \right\rbrace  \text{ \quad by \eqref{eqt2} and \eqref{eqt*2}}.
		\end{align*} 
	\end{proof}
	
We present next a lower bound for the generalized Davis--Wielandt radius of bounded
operators for any norm on $B(\mathcal{H})$.
	\begin{theorem}\label{The3}
		Let $T \in B(\mathcal{H} ) $ and $N(\cdot)$ be a norm. Then  
		\begin{equation}\label{eq1th3}
			dw_{N}(T) \geq   \dfrac{ 1}{2} \sqrt{ \dfrac{3}{4}N^2(T) +  2N^4(|T|) 
				+ (d_1 + d_2) +2 |m_1-m_2|}, 
		\end{equation} 
		where 
		\begin{align*}
			m_1 =& \max\{ N^2(\Re (T))+N^4(|T|), w^2_N(T)  \}; \\
			m_2 =& \max\{ N^2(\Im (T)) , N^4(|T|)  \}; \\			
			d_1 =& | N^2(\Re (T))+N^4(|T|)- w^2_N(T)   |; \\
			d_2 =& | N^2(\Im (T))- N^4(|T|)   |. 
		\end{align*}
		
	\end{theorem}
	\begin{proof} 
		We know that 
		\begin{align*}
			dw_{N}(T) =&  \sup_{\theta \in \mathbb{R}}\sqrt{N^2(\Re (e^{i\theta}T))+N^4(\Re(e^{i\theta}|T|))}\\
			\geq & \max \left\lbrace \sup_{\theta \in \mathbb{R}}\sqrt{  N^2(\Re (e^{i\theta}T))}, \sup_{\theta \in \mathbb{R}}\sqrt{  N^4(\Re(e^{i\theta}|T|))}\right\rbrace    \\
			=& \max \left\lbrace w_N(T),  N^2(|T|)\right\rbrace,
		\end{align*} 
		and, we see that 	
		\begin{equation*} 
			dw^2_{N}(T) \geq   {N^2(\Re (T))+N^4(|T|) }
		\end{equation*} 
		and 
		\begin{equation*} 
			dw^2_{N}(T) \geq   {N^2(\Im (T)) }. 
		\end{equation*}
		Using, the well known inequality $w_N(T) \geq \dfrac{1}{2} N(T) $, we obtain
		\begin{align}
			dw^2_{N}(T) \geq  & \max \{ m_1 , m_2 \}  \nonumber \\
			= &  \dfrac{1}{2}( m_1+ m_2 +|m_1-m_2|)  \nonumber \\
			= &  \dfrac{1}{4}\left( N^2(\Re (T)) + N^4(|T|) + w^2_N(T) + d_1       \right) \nonumber \\
			+& \dfrac{1}{4}\left( N^2(\Im (T))  +N^4(|T|) + d_2       \right) + \dfrac{1}{2} |m_1-m_2| \nonumber \\
			=& \dfrac{1}{4}\left( N^2(\Re (T)) + N^2(\Im (T)) + w^2_N(T)  + (d_1 + d_2)\right) \nonumber \\
			+& \dfrac{1}{2} \left( |m_1-m_2|
			+ N^4(|T|) \right) \label{eqp2}\\
			\geq & \dfrac{1}{4} \left( \dfrac{1}{2}N^2(T) +  w^2_N(T)  + (d_1 + d_2) \right)   \text{ \quad by \eqref{eqt1}}\nonumber\\ 
			+ &  \dfrac{1}{2} \left( N^4(|T|)
			+  |m_1-m_2| \right) \nonumber \\
			\geq & \dfrac{1}{4} \left( \dfrac{1}{2}N^2(T) +  \dfrac{1}{4}N^2(T)  + (d_1 + d_2) \right)   \nonumber\\ 
			+ &  \dfrac{1}{2} \left(  N^4(|T|)
			+  |m_1-m_2|\right)  \nonumber \\
			\geq &  \dfrac{1}{4} \left( \dfrac{3}{4}N^2(T) +  2N^4(|T|) 
			+ (d_1 + d_2) +2 |m_1-m_2|\right), \nonumber 
		\end{align}
		this completes the proof.
	\end{proof}
	\begin{theorem}
		Let $T \in B(\mathcal{H} ) $ and $N(\cdot)$ be an algebra norm. Then  
		\begin{equation}\label{eq+1}
			dw^2_{N}(T) 
			\geq  \dfrac{1}{2}  \sqrt{\dfrac{ 1}{2}N(|T|^2 + |T^*|^2 ) + w^2_N(T) +  2N^4(|T|) 
				+ (d_1 + d_2) +2 |m_1-m_2|}, 
		\end{equation}
		and 
		\begin{equation}\label{eq+2}
			dw^2_{N}(T) 
			\geq  \dfrac{1}{2}  \sqrt{\dfrac{ 1}{2}N(T^2 + T^{*2} ) + w^2_N(T) +  2N^4(|T|) 
				+ (d_1 + d_2) +2 |m_1-m_2|}, 
		\end{equation}
		
		where $m_1$, $m_2$, $d_1$ and $d_2$ are given in Theorem \ref{The3}.  
	\end{theorem}
	
	\begin{proof}
		It follows from \eqref{eqp2} that
		\begin{align*}
			dw^2_{N}(T) \geq   & \dfrac{1}{4}( N^2(\Re (T)) + N^2(\Im (T)) + w^2_N(T)  + (d_1 + d_2)) \nonumber \\
			& + \dfrac{1}{2} \left( |m_1-m_2|
			+ N^4(|T|) \right) \\
			& \geq \dfrac{1}{4}\left(  \dfrac{ 1}{2}N(|T|^2 + |T^*|^2 ) + w^2_N(T)  + (d_1 + d_2)\right)  \text{ \quad \quad by \eqref{eqt2}}  \\
			& +   \dfrac{1}{2} \left( |m_1-m_2| 
			+ N^4(|T|) \right),
		\end{align*}
		which gives \eqref{eq+1}.\\
		Similarly, using \eqref{eqt*2}, we prove \eqref{eq+2}.
	\end{proof}

	
	\begin{remark}
		Example \ref{nilp}  shows again the sharpness of estimations in \eqref{eq1th3},  \eqref{eq+1} and \eqref{eq+2} as follows, 
		\begin{align*}
			m_1 =& \max\{ \| \Re (T)\|^2+\||T|\|^4, w^2(T)  \} =17,  \\
			m_2 =& \max\{ \|\Im (T)\|^2 ,\||T|\|^4  \}= 16, \\			
			d_1 =& | \|\Re (T)\|^2 +\||T|\|^4 - w^2(T)   | = 16, \\
			d_2 =& | \|\Im (T)\|^2- \||T|\|^4   | = 15. 
		\end{align*}
		Therefore
		\begin{align*}
			&\dfrac{ 1}{2} \sqrt{ \dfrac{3}{4}N^2(T) +  2N^4(|T|) 
				+ (d_1 + d_2) +2 |m_1-m_2|} \text{ \quad given in \eqref{eq1th3}} \\
			& = \dfrac{ 1}{2} \sqrt{\dfrac{ 1}{2}N(|T|^2 + |T^*|^2 ) + w^2_N(T) +  2N^4(|T|) 
				+ (d_1 + d_2) +2 |m_1-m_2|}     \\
			& \text{ \quad \quad\quad\quad\quad\quad \quad \quad\quad\quad\quad\quad \quad \quad\quad\quad\quad\quad \quad \quad given in \eqref{eq+1}}\\
			&  =\dfrac{1}{2}  \sqrt{\dfrac{ 1}{2}N(T^2 + T^{*2} ) + w^2_N(T) +  2N^4(|T|) 
				+ (d_1 + d_2) +2 |m_1-m_2|} \\
			& \text{ \quad \quad\quad\quad\quad\quad \quad \quad\quad\quad\quad\quad \quad \quad\quad\quad\quad\quad \quad \quad given in \eqref{eq+2}}\\
			& = \sqrt{17}.
		\end{align*}
	\end{remark}
	We note here that, if $N(\cdot)$ is self--adjoint norm so $w_N(T) \leq N(T)$ (see, \cite[Theorem 2]{Abu}), and if we consider \eqref{eqp2}, we obtain 
	\begin{align}
		dw^2_{N}(T) &\geq   \dfrac{1}{4}\left( N^2(\Re (T)) + N^2(\Im (T)) +w^2_N(T)  + (d_1 + d_2) \right)  \nonumber \\
		&+ \dfrac{1}{2} \left( |m_1-m_2|
		+ N^4(|T|)\right)  \nonumber \\
		&\geq  \dfrac{1}{4} \left( \dfrac{1}{2}N^2(T) +w^2_N(T)  +  (d_1 + d_2) \right)  \nonumber\\
		& + \dfrac{1}{2} \left(  |m_1-m_2|+ N^4(|T|) \right) \nonumber \\
		&\geq  \dfrac{1}{4} \left( \dfrac{1}{2}w^2_N(T) +w^2_N(T)  +  (d_1 + d_2) \right)  \nonumber\\
		& + \dfrac{1}{2} \left(  |m_1-m_2|+ N^4(|T|) \right) \nonumber \\
		&\geq   \dfrac{1}{4} \left(  \dfrac{3}{2}w^2_N(T)  +  2N^4(|T|) 
		+ (d_1 + d_2) +2 |m_1-m_2|\right). \nonumber \\
	\end{align} 
	So, we have the following result

	\begin{theorem}\label{The4}
		Let $T \in B(\mathcal{H} ) $ and $N(\cdot)$ be a self--adjoint norm. Then  
		\begin{equation}
			dw_{N}(T) \geq   \dfrac{ 1}{2} \sqrt{ \dfrac{3}{2}w^2_N(T) +  2N^4(|T|) 
				+ (d_1 + d_2) +2 |m_1-m_2|}, 
		\end{equation} 
		where $m_1$, $m_2$, $d_1$ and $d_2$ are given in Theorem \ref{The3}.  
	\end{theorem}
	In general, the generalized Davis--Wielandt radius fails to satisfy the triangular inequality. This result motivates the development of an alternative operator sum inequality, which is presented in the following  theorem.
	\begin{theorem}
		Let $T,S \in B(\mathcal{H} ) $ and $N(\cdot)$ be a norm. Then	
		\begin{align*}
			dw_{N}(T + S ) & \leq  \sqrt{2(dw_{N}^2(T ) + dw_{N}^2(S )) + 6( N^4(|T|) + N^4(|S|))}\\
			&  \leq 2\sqrt{2}\sqrt{dw_{N}^2(T ) + dw_{N}^2(S )}\\
			&  \leq 2\sqrt{2}(dw_{N}(T ) + dw_{N}(S ))
		\end{align*} 
	\end{theorem}
	\begin{proof}
		For any $\theta \in \mathbb{R}$, we have 
		\begin{align*}
			& N^2(\Re (e^{i\theta}T + e^{i\theta}S ))+N^4(\Re(e^{i\theta}|T| + e^{i\theta}|S|))\\
			& \leq  \left(  N(\Re (e^{i\theta}T )) + N(\Re (e^{i\theta}S )) \right)^2 +   \left( N(\Re(e^{i\theta}|T|)) + N(\Re(e^{i\theta}|S|))\right)^4\\
			& \leq  2\left( N^2(\Re (e^{i\theta}T )) + N^2(\Re (e^{i\theta}S )) \right) + 4 \left( N^2(\Re(e^{i\theta}|T|)) + N^2(\Re(e^{i\theta}|S|))\right)^2\\
			&\leq  2\left( N^2(\Re (e^{i\theta}T )) + N^2(\Re (e^{i\theta}S )) \right) + 8 \left( N^4(\Re(e^{i\theta}|T|)) + N^4(\Re(e^{i\theta}|S|))\right)\\
			& \leq 2\left( N^2(\Re (e^{i\theta}T )) + N^4(\Re(e^{i\theta}|T|)) \right)  + 2\left( N^2(\Re (e^{i\theta}S )) + N^4(\Re(e^{i\theta}|S|)) \right)\\
			&+ 6 \left( N^4(\Re(e^{i\theta}|T|)) + N^4(\Re(e^{i\theta}|S|))\right)\\
			& \leq  2(dw_{N}^2(T ) + dw_{N}^2(S )) + 6( N^4(|T|) + N^4(|S|))\\
			&  \leq 8(dw_{N}^2(T ) + dw_{N}^2(S ))\\
			&  \leq 8(dw_{N}(T ) + dw_{N}(S ))^2.
		\end{align*} 
		Taking supremum over all $\theta \in \mathbb{R}$, we get the desire inequalities.
	\end{proof}
	\section*{Acknowledgements}
	The authors gratefully acknowledge the financial support from the Laboratory of Fundamental and Applicable Mathematics of Oran (LMFAO) and  the Algerian research project: PRFU, No: C00L03ES310120220003 and PRFU, No: C00L03ES310120220001 (D.G.R.S.D.T).
	
	\textbf{Conflict of interest }The authors declare that they have no conflict of interest.


\begin{thebibliography}{10}
	\ifx \bisbn   \undefined \def \bisbn  #1{ISBN #1}\fi
	\ifx \binits  \undefined \def \binits#1{#1}\fi
	\ifx \bauthor  \undefined \def \bauthor#1{#1}\fi
	\ifx \batitle  \undefined \def \batitle#1{#1}\fi
	\ifx \bjtitle  \undefined \def \bjtitle#1{#1}\fi
	\ifx \bvolume  \undefined \def \bvolume#1{\textbf{#1}}\fi
	\ifx \byear  \undefined \def \byear#1{#1}\fi
	\ifx \bissue  \undefined \def \bissue#1{#1}\fi
	\ifx \bfpage  \undefined \def \bfpage#1{#1}\fi
	\ifx \blpage  \undefined \def \blpage #1{#1}\fi
	\ifx \burl  \undefined \def \burl#1{\textsf{#1}}\fi
	\ifx \doiurl  \undefined \def \doiurl#1{\url{https://doi.org/#1}}\fi
	\ifx \betal  \undefined \def \betal{\textit{et al.}}\fi
	\ifx \binstitute  \undefined \def \binstitute#1{#1}\fi
	\ifx \binstitutionaled  \undefined \def \binstitutionaled#1{#1}\fi
	\ifx \bctitle  \undefined \def \bctitle#1{#1}\fi
	\ifx \beditor  \undefined \def \beditor#1{#1}\fi
	\ifx \bpublisher  \undefined \def \bpublisher#1{#1}\fi
	\ifx \bbtitle  \undefined \def \bbtitle#1{#1}\fi
	\ifx \bedition  \undefined \def \bedition#1{#1}\fi
	\ifx \bseriesno  \undefined \def \bseriesno#1{#1}\fi
	\ifx \blocation  \undefined \def \blocation#1{#1}\fi
	\ifx \bsertitle  \undefined \def \bsertitle#1{#1}\fi
	\ifx \bsnm \undefined \def \bsnm#1{#1}\fi
	\ifx \bsuffix \undefined \def \bsuffix#1{#1}\fi
	\ifx \bparticle \undefined \def \bparticle#1{#1}\fi
	\ifx \barticle \undefined \def \barticle#1{#1}\fi
	\bibcommenthead
	\ifx \bconfdate \undefined \def \bconfdate #1{#1}\fi
	\ifx \botherref \undefined \def \botherref #1{#1}\fi
	\ifx \url \undefined \def \url#1{\textsf{#1}}\fi
	\ifx \bchapter \undefined \def \bchapter#1{#1}\fi
	\ifx \bbook \undefined \def \bbook#1{#1}\fi
	\ifx \bcomment \undefined \def \bcomment#1{#1}\fi
	\ifx \oauthor \undefined \def \oauthor#1{#1}\fi
	\ifx \citeauthoryear \undefined \def \citeauthoryear#1{#1}\fi
	\ifx \endbibitem  \undefined \def \endbibitem {}\fi
	\ifx \bconflocation  \undefined \def \bconflocation#1{#1}\fi
	\ifx \arxivurl  \undefined \def \arxivurl#1{\textsf{#1}}\fi
	\csname PreBibitemsHook\endcsname
	
	\bibitem[\protect\citeauthoryear{Abu-Omar and Kittaneh}{2019}]{Abu}
	\begin{barticle}
		\bauthor{\bsnm{Abu--Omar}, \binits{A.}},
		\bauthor{\bsnm{Kittaneh}, \binits{F.}}:
		\batitle{A generalization of the numerical radius}.
		\bjtitle{Linear Algebra Appl.}
		\bvolume{569},
		\bfpage{323}--\blpage{334}
		(\byear{2019})
		\doiurl{10.1016/j.laa.2019.01.019}
	\end{barticle}
	\endbibitem
	
	\bibitem[\protect\citeauthoryear{Alomari et~al.}{2024a}]{Alo}
	\begin{barticle}
		\bauthor{\bsnm{Alomari}, \binits{M.W.}},
		\bauthor{\bsnm{Bakherad}, \binits{M.}},
		\bauthor{\bsnm{Hajmohamadi}, \binits{M.}}:
		\batitle{A generalization of the Davis--Wielandt radius for operators}.
		\bjtitle{Bolet\'{i}n Soc. Mat. Mex.}
		\bvolume{30}(\bissue{2}),
		\bfpage{57}
		(\byear{2024})
		\doiurl{10.1007/s40590-024-00631-6}
	\end{barticle}
	\endbibitem
	
	\bibitem[\protect\citeauthoryear{Alomari et~al.}{2024b}]{Alo1}
	\begin{barticle}
		\bauthor{\bsnm{Alomari}, \binits{M.W.}},
		\bauthor{\bsnm{Sababheh}, \binits{M.}},
		\bauthor{\bsnm{Conde}, \binits{C.}},
		\bauthor{\bsnm{Moradi}, \binits{H.R.}}:
		\batitle{Generalized Euclidean operator radius}.
		\bjtitle{Georg. Math. J.}
		\bvolume{31}(\bissue{3}),
		\bfpage{369}--\blpage{380}
		(\byear{2024})
		\doiurl{10.1515/gmj-2023-2079}
	\end{barticle}
	\endbibitem
	
	\bibitem[\protect\citeauthoryear{Bakherad et~al.}{2023}]{Bak}
	\begin{barticle}
		\bauthor{\bsnm{Bakherad}, \binits{M.}},
		\bauthor{\bsnm{Hajmohamadi}, \binits{M.}},
		\bauthor{\bsnm{Lashkaripour}, \binits{R.}},
		\bauthor{\bsnm{Sababheh}, \binits{M.}}:
		\batitle{Some estimates for the generalized numerical radius}.
		\bjtitle{J. Anal.}
		\bvolume{31},
		\bfpage{2163}--\blpage{2172}
		(\byear{2023})
		\doiurl{10.1007/s41478-023-00557-8}
	\end{barticle}
	\endbibitem
		\bibitem[\protect\citeauthoryear{Bhanja et al.}{2021}]{Bha1}
	\begin{barticle}
		\bauthor{\bsnm{Bhanja}, \binits{A.}},
		\bauthor{\bsnm{Bhunia}, \binits{P.}},
		\bauthor{\bsnm{Paul}, \binits{K.}}:
		\batitle{On generalized Davis--Wielandt radius inequalities of semi-Hilbertian space operators}.
		\bjtitle{Oper. Matrices}
		\bvolume{15}(\bissue{4}),
		\bfpage{1201}--\blpage{1225}
		(\byear{2021})
		\doiurl{10.7153/oam-2021-15-76}
	\end{barticle}
	\endbibitem

	
	\bibitem[\protect\citeauthoryear{Bhunia et~al.}{2021}]{Bhu1}
	\begin{barticle}
		\bauthor{\bsnm{Bhunia}, \binits{P.}},
		\bauthor{\bsnm{Bhanja}, \binits{A.}},
		\bauthor{\bsnm{Bag}, \binits{S.}},
		\bauthor{\bsnm{Paul}, \binits{K.}}:
		\batitle{Bounds for the Davis--Wielandt radius of bounded linear operators}.
		\bjtitle{Ann. Funct. Anal.}
		\bvolume{12}(\bissue{1}),
		\bfpage{18}
		(\byear{2021})
		\doiurl{10.1007/s43034-020-00102-9}
	\end{barticle}
	\endbibitem
	
	\bibitem[\protect\citeauthoryear{Bhunia et al.}{2021}]{Bhu---r}
	\begin{barticle}
		\bauthor{\bsnm{Bhunia}, \binits{P.}},
		\bauthor{\bsnm{Bhanja}, \binits{A.}},
		\bauthor{\bsnm{Paul}, \binits{K.}}:
		\batitle{New inequalities for Davis--Wielandt radius of Hilbert space operators}.
		\bjtitle{Bull. Malays. Math. Sci. Soc.}
		\bvolume{44}(\bissue{5}),
		\bfpage{3523}--\blpage{3539}
		(\byear{2021})
		\doiurl{10.1007/s40840-021-01126-7}
	\end{barticle}
	\endbibitem
	
		\bibitem[\protect\citeauthoryear{Bhunia et al.}{2022}]{BhuBook}
	\begin{bbook}
		\bauthor{\bsnm{Bhunia}, \binits{P.}},
		\bauthor{\bsnm{Dragomir}, \binits{S.S.}},
		\bauthor{\bsnm{Moslehian}, \binits{M.S.}},
		\bauthor{\bsnm{Paul}, \binits{K.}}:
		\bbtitle{Lectures on Numerical Radius Inequalities}.
		\bpublisher{Springer},
		\blocation{Cham},
		\byear{2022}.
		\doiurl{10.1007/978-3-031-13670-2}
	\end{bbook}
	\endbibitem
	\bibitem[\protect\citeauthoryear{Bhunia et~al.}{2023}]{Bhu}
	\begin{barticle}
		\bauthor{\bsnm{Bhunia}, \binits{P.}},
		\bauthor{\bsnm{Paul}, \binits{K.}},
		\bauthor{\bsnm{Barik}, \binits{S.}}:
		\batitle{Further refinements of Davis--Wielandt radius inequalities}.
		\bjtitle{Oper. Matrices}
		\bvolume{17}(\bissue{3}),
		\bfpage{737}--\blpage{748}
		(\byear{2023})
		\doiurl{10.7153/oam-2023-17-50}
	\end{barticle}
	\endbibitem
	
	\bibitem[\protect\citeauthoryear{Bottazzi and Conde}{2021}]{Bot}
	\begin{barticle}
		\bauthor{\bsnm{Bottazzi}, \binits{T.}},
		\bauthor{\bsnm{Conde}, \binits{C.}}:
		\batitle{Generalized numerical radius and related inequalities}.
		\bjtitle{Oper. Matrices}
		\bvolume{15}(\bissue{4}),
		\bfpage{1289}--\blpage{1308}
		(\byear{2021})
		\doiurl{10.7153/oam-2021-15-81}
	\end{barticle}
	\endbibitem
	
	\bibitem[\protect\citeauthoryear{Davis}{1968}]{Dav}
	\begin{barticle}
		\bauthor{\bsnm{Davis}, \binits{C.}}:
		\batitle{The shell of a Hilbert-space operator}.
		\bjtitle{Acta Sci. Math.(Szeged)}
		\bvolume{29}(\bissue{1-2}),
		\bfpage{69}--\blpage{86}
		(\byear{1968})
	\end{barticle}
	\endbibitem
	
	\bibitem[\protect\citeauthoryear{Hajmohamadi et~al.}{2021a}]{Haj}
	\begin{barticle}
		\bauthor{\bsnm{Hajmohamadi}, \binits{M.}},
		\bauthor{\bsnm{Lashkaripour}, \binits{R.}},
		\bauthor{\bsnm{Bakherad}, \binits{M.}}:
		\batitle{Further refinements of generalized numerical radius inequalities for
			Hilbert space operators}.
		\bjtitle{Georg. Math. J.}
		\bvolume{28}(\bissue{1}),
		\bfpage{83}--\blpage{92}
		(\byear{2021})
		\doiurl{10.1515/gmj-2019-2023}
	\end{barticle}
	\endbibitem
	

	
	\bibitem[\protect\citeauthoryear{Karaev and Iskenderov}{2010}]{Kar}
	\begin{botherref}
		\bauthor{\bsnm{Karaev}, \binits{M.}},
		\bauthor{\bsnm{Iskenderov}, \binits{N.S.}}:
		\batitle{Numerical Range and Numerical Radius for Some Operators}.
		\bjtitle{Linear Algebra Appl.}
		\bvolume{432}(\bissue{12}),
		\bfpage{3149}--\blpage{3158}
		(\byear{2010})
		\doiurl{10.1016/j.laa.2010.01.011}
	\end{botherref}
	\endbibitem
	

	\bibitem[\protect\citeauthoryear{Kittaneh and Zamani}{2024}]{Kit}
	\begin{barticle}
	\bauthor{\bsnm{Kittaneh}, \binits{F.}},
		\bauthor{\bsnm{Zamani}, \binits{A.}}:
		\batitle{	On the $\rho$-operator radii}.
		\bjtitle{Linear Algebra Appl.}
		\bvolume{687},
		\bfpage{132}--\blpage{156}
		(\byear{2024})
		\doiurl{10.1016/j.laa.2024.01.023}
	\end{barticle}
	\endbibitem
	
	
	\bibitem[\protect\citeauthoryear{Kittaneh}{2003}]{Kit1}
	\begin{barticle}
		\bauthor{\bsnm{Kittaneh}, \binits{F.}}:
		\batitle{A numerical radius inequality and an estimate for the numerical radius
			of the frobenius companion matrix}.
		\bjtitle{Studia Math.}
		\bvolume{158}(\bissue{1}),
		\bfpage{11}--\blpage{17}
		(\byear{2003})
	\end{barticle}
	\endbibitem
	
	\bibitem[\protect\citeauthoryear{Saddi}{2012}]{Sad}
	\begin{barticle}
		\bauthor{\bsnm{Saddi}, \binits{A.}}:
		\batitle{A-normal operators in semi-Hilbertian spaces}.
		\bjtitle{Aust. J. Math. Anal. Appl.}
		\bvolume{9}(\bissue{1}),
		\bfpage{1}--\blpage{12}
		(\byear{2012})
	\end{barticle}
	\endbibitem
	
	\bibitem[\protect\citeauthoryear{Wielandt et~al.}{1955}]{Wie}
	\begin{barticle}
		\bauthor{\bsnm{Wielandt}, \binits{H.}}, \betal:
		\batitle{On eigenvalues of sums of normal matrices}.
		\bjtitle{Pac. J. Math}
		\bvolume{5}(\bissue{4}),
		\bfpage{633}--\blpage{638}
		(\byear{1955})
	\end{barticle}
	\endbibitem
		\bibitem[\protect\citeauthoryear{Yamazaki}{2007b}]{Yam}
	\begin{barticle}
		\bauthor{\bsnm{Yamazaki}, \binits{T.}}:
		\batitle{On upper and lower bounds of the numerical radius and an equality condition}.
		\bjtitle{Studia Math.}
		\bvolume{178}(\bissue{1}),
		\bfpage{83}--\blpage{89}
		(\byear{2007})
	\end{barticle}
	\endbibitem
	
	\bibitem[\protect\citeauthoryear{Zamani et~al.}{2019}]{Zam2}
	\begin{barticle}
		\bauthor{\bsnm{Zamani}, \binits{A.}},
		\bauthor{\bsnm{Moslehian}, \binits{M.S.}},
		\bauthor{\bsnm{Chien}, \binits{M.-T.}},
		\bauthor{\bsnm{Nakazato}, \binits{H.}}:
		\batitle{Norm-parallelism and the Davis--Wielandt radius of Hilbert space operators}.
		\bjtitle{Linear Algebra Appl.}
		\bvolume{67}(\bissue{11}),
		\bfpage{2147}--\blpage{2158}
		(\byear{2019})
		\doiurl{10.1080/03081087.2018.1484422}
	\end{barticle}
	\endbibitem
	
	\bibitem[\protect\citeauthoryear{Zamani and Shebrawi}{2020}]{Zam}
	\begin{barticle}
		\bauthor{\bsnm{Zamani}, \binits{A.}},
		\bauthor{\bsnm{Shebrawi}, \binits{K.}}:
		\batitle{Some upper bounds for the Davis--Wielandt radius of Hilbert space
			operators}.
		\bjtitle{Mediterr. J. Math.}
		\bvolume{17}(\bissue{1}),
		\bfpage{25}
		(\byear{2020})
		\doiurl{10.1007/s00009-019-1458-z}
	\end{barticle}
	\endbibitem
	
	\bibitem[\protect\citeauthoryear{Zamani and W{\'o}jcik}{2021}]{Zam1}
	\begin{barticle}
		\bauthor{\bsnm{Zamani}, \binits{A.}},
		\bauthor{\bsnm{W{\'o}jcik}, \binits{P.}}:
		\batitle{Another generalization of the numerical radius for Hilbert space
			operators}.
		\bjtitle{Linear Algebra Appl.}
		\bvolume{609},
		\bfpage{114}--\blpage{128}
		(\byear{2021})
		\doiurl{10.1016/j.laa.2020.08.032}
	\end{barticle}
	\endbibitem
	
\end{thebibliography}

\end{document}